\documentclass[11pt,reqno]{amsart}

\usepackage{amssymb}
\usepackage{amsfonts}
\usepackage{amsthm}

\theoremstyle{plain} 
\newtheorem{theorem}{Theorem}

\newtheorem{lemma}[theorem]{Lemma}
\theoremstyle{definition} 
\newtheorem{definition}[theorem]{Definition}
\newtheorem{remark}[theorem]{Remark}

\newcommand{\R}{\ensuremath{\mathbb{R}}}

\newcommand{\T}{\ensuremath{\mathbb{T}}}
\newcommand{\N}{\ensuremath{\mathbb{N}}}
\newcommand{\C}{\ensuremath{\mathbb{C}}}

\DeclareMathOperator{\czero}{C_{rd}}
\DeclareMathOperator{\crdone}{C^{\Delta}_{rd}}
\DeclareMathOperator{\crdtwo}{C^{\Delta^2}_{rd}}

\numberwithin{equation}{section}
\numberwithin{theorem}{section}

\small\normalsize

\begin{document}
\author[anderson]{Douglas R. Anderson}
\title[hyers--ulam stability on time scales]{Hyers--Ulam stability of higher-order Cauchy-Euler dynamic equations on time scales}
\address{Department of Mathematics, Concordia College, Moorhead, MN 56562 USA}
\email{andersod@cord.edu, bgates@cord.edu, dheuer@cord.edu}
\urladdr{http://www.cord.edu/faculty/andersod/}


\begin{abstract}
We establish the stability of higher-order linear nonhomogeneous Cauchy-Euler dynamic equations on time scales in the sense of Hyers and Ulam. That is, if an approximate solution of a higher-order Cauchy-Euler equation exists, then there exists an exact solution to that dynamic equation that is close to the approximate one. Some examples illustrate the applicability of the main results.
\end{abstract}

\maketitle

\thispagestyle{empty}


\section{introduction}
In 1940, Ulam \cite{ulam} posed the following problem concerning the stability of functional equations: give conditions in order for a linear mapping near an approximately linear mapping to exist. The problem for the case of approximately additive mappings was solved by Hyers \cite{hyers} who proved that the Cauchy equation is stable in Banach spaces, and the result of Hyers was generalized by Rassias \cite{rassias}. Alsina and Ger \cite{ag} were the first authors who investigated the Hyers-Ulam stability of a differential equation.

Since then there has been a significant amount of interest in Hyers-Ulam stability, especially in relation to ordinary differential equations, for example see \cite{gavruta,gjl,jung,jung2,jung3,jung4,mmt,mmt2,pr,rus}. Also of interest are many of the articles in a special issue guest edited by Rassias \cite{rass2}, dealing with Ulam, Hyers-Ulam, and Hyers-Ulam-Rassias stability in various contexts. Also see Li and Shen \cite{lishen1,lishen2}, Wang, Zhou, and Sun \cite{wang}, and Popa et al \cite{popa, bpx}.  Andr\'{a}s and M\'{e}sz\'{a}ros \cite{andras} recently used an operator approach to show the stability of linear dynamic equations on time scales with constant coefficients, as well as for certain integral equations. Anderson et al \cite[Corollary 2.6]{and} proved the following concerning second-order non-homogeneous Cauchy-Euler equations on time scales:


\begin{theorem}[Cauchy-Euler Equation]\label{coroCE}
Let $\lambda_1,\lambda_2\in\R$ (or $\lambda_2=\overline{\lambda_1}$, the complex conjugate) be such that 
$$ t+\lambda_k \mu(t) \ne 0, \quad k=1,2 $$
for all $t\in[a,\sigma(b)]_{\T}$, where $a\in\T$ satisfies $a>0$. Then the Cauchy-Euler equation
\begin{equation}\label{CEeq}
  x^{\Delta\Delta}(t) + \frac{1-\lambda_1-\lambda_2}{\sigma(t)} \; x^\Delta(t) + \frac{\lambda_1\lambda_2}{t\sigma(t)} \; x(t) = f(t), \quad t\in[a,b]_\T
\end{equation}
has Hyers-Ulam stability on $[a,b]_{\T}$. To wit, if there exists $y\in\crdtwo[a,b]_\T$ that satisfies 
$$ \left| y^{\Delta\Delta}(t) + \frac{1-\lambda_1-\lambda_2}{\sigma(t)} \; y^\Delta(t) + \frac{\lambda_1\lambda_2}{t\sigma(t)} \; y(t) - f(t) \right| \le \varepsilon $$
for $t\in[a,b]_\T$, then there exists a solution $u\in\crdtwo[a,b]{_\T}$ of \eqref{CEeq} given by
$$ u(t) = e_{\frac{\lambda_1}{t}}\left(t,\tau_2\right)y\left(\tau_2\right)+\int_{\tau_2}^t e_{\frac{\lambda_1}{t}}(t,\sigma(s))w(s)\Delta s, \quad\text{any}\quad \tau_2\in[a,\sigma^2(b)]_{\T}, $$
where for any $\tau_1\in[a,\sigma(b)]_{\T}$ the function $w$ is given by
$$ w(s) = e_{\frac{\lambda_2-1}{\sigma(s)}}(s,\tau_1) \left[y^\Delta(\tau_1)-\frac{\lambda_1}{\tau_1}y(\tau_1)\right] + \int_{\tau_1}^s e_{\frac{\lambda_2-1}{\sigma(s)}}(s,\sigma(\zeta)) f(\zeta) \Delta \zeta, $$
such that $|y-u|\le K\varepsilon$ on $[a,\sigma^2(b)]_\T$ for some constant $K>0$.
\end{theorem}

The motivation for this work is to extend Theorem \ref{coroCE} to the general $n$th-order Cauchy-Euler dynamic equation; we will show the stability in the sense of Hyers and Ulam of the equation
$$ \sum_{k=0}^n \alpha_k M_k y(t)=f(t), \quad M_0y(t):=y(t), \quad M_{k+1}y(t):=\varphi(t)\left(M_ky\right)^\Delta(t),\; k=0,1,\cdots,n-1. $$
This is essentially \cite[(2.14)]{bp2} if $\varphi(t)=t$ and $f(t)=0$.
Throughout this work we assume the reader has a working knowledge of time scales as can be found in Bohner and Peterson \cite{bp1,bp2}, originally introduced by Hilger \cite{hilger}.


\section{Hyers-Ulam stability for higher-order Cauchy-Euler dynamic equations}

In this section we establish the Hyers-Ulam stability of the higher-order non-homogeneous Cauchy-Euler dynamic equation on time scales of the form
\begin{equation}\label{2ndivc}
 \sum_{k=0}^n \alpha_k M_k y(t)=f(t), \quad M_0y(t):=y(t), \quad M_{k+1}y(t):=\varphi(t)\left(M_ky\right)^\Delta(t),\; k=0,1,\cdots,n-1
\end{equation}
for given constants $\alpha_k\in\R$ with $\alpha_n\equiv 1$, and for functions $\varphi,f\in\czero[a,b]_\T$, using the following definition.


\begin{definition}[Hyers-Ulam stability]
Let $\varphi,f\in\czero[a,b]_\T$ and $n\in\N$. If whenever $M_kx\in\crdone[a,b]{_\T}$ satisfies
$$ \left|  \sum_{k=0}^n \alpha_k M_k x(t) - f(t) \right|\le \varepsilon, \quad t\in[a,b]_{\T} $$
there exists a solution $u$ of \eqref{2ndivc} with $M_ku\in\crdone[a,b]{_\T}$ for $k=0,1,\cdots,n-1$ such that $|x-u|\le K\varepsilon$ on $[a,\sigma^n(b)]_\T$ for some constant $K>0$, then \eqref{2ndivc} has Hyers-Ulam stability $[a,b]_\T$.
\end{definition}


\begin{remark}
Before proving the Hyers-Ulam stability of \eqref{2ndivc} we will need the following lemma, which allows us to factor \eqref{2ndivc} using the elementary symmetric polynomials \cite{pbg} in the $n$ symbols $\rho_1,\cdots,\rho_n$ given by
\begin{eqnarray*}
 s_1^n &=& s_1(\rho_1,\cdots,\rho_n) = \sum_i \rho_i \\
 s_2^n &=& s_2(\rho_1,\cdots,\rho_n) = \sum_{i<j} \rho_i\rho_j \\
 s_3^n &=& s_3(\rho_1,\cdots,\rho_n) = \sum_{i<j<k} \rho_i\rho_j\rho_k \\  
 s_4^n &=& s_4(\rho_1,\cdots,\rho_n) = \sum_{i<j<k<\ell} \rho_i\rho_j\rho_k\rho_{\ell} \\
  &\vdots& \\
 s_t^n &=& s_t(\rho_1,\cdots,\rho_n) = \sum_{i_1<i_2<\cdots<i_t} \rho_{i_1}\rho_{i_2}\cdots\rho_{i_t} \\
  &\vdots& \\
 s_n^n &=& s_n(\rho_1,\cdots,\rho_n) = \rho_{1}\rho_{2}\rho_3\cdots\rho_{n}.
\end{eqnarray*}
In general, we let $s_i^j$ represent the $i$th elementary symmetric polynomial on $j$ symbols. Then, given the $\alpha_k$ in \eqref{2ndivc}, introduce the characteristic values $\lambda_k\in\C$ via the elementary symmetric polynomial $s_t^n$ on the $n$ symbols $-\lambda_1,\cdots,-\lambda_n$, where 
\begin{equation}\label{alphas}
 \alpha_k=s_{n-k}^n=s_{n-k}(-\lambda_1,\cdots,-\lambda_n)= \sum_{i_1<i_2<\cdots<i_{n-k}} (-1)^{n-k}\lambda_{i_1}\lambda_{i_2}\cdots\lambda_{i_{n-k}}, \quad \alpha_n=s_0\equiv 1.
\end{equation}

\end{remark}


\begin{lemma}[Factorization]\label{lemma2.3}
Given $y,\varphi\in\czero[a,b]_\T$ and $\alpha_k\in\R$ with $\alpha_n\equiv 1$, let $M_ky\in\crdone[a,b]{_\T}$, where $M_0y(t):=y(t)$ and $M_{k+1}y(t):=\varphi(t)\left(M_ky\right)^\Delta(t)$ for $k=0,1,\cdots,n-1$. Then we have the factorization 
\begin{equation}\label{factorize}
 \sum_{k=0}^n \alpha_k M_k y(t)=\prod_{k=1}^{n} \left(\varphi D-\lambda_kI\right)y(t), \qquad n\in\N,
\end{equation}
where the differential operator $D$ is defined via $Dx=x^\Delta$ for $x\in\crdone[a,b]_\T$, and $I$ is the identity operator.
\end{lemma}

\begin{proof}
We proceed by mathematical induction on $n\in\N$, utilizing the substitution defined in \eqref{alphas}. For $n=1$, 
$$ \sum_{k=0}^n \alpha_k M_k y(t) = \alpha_0 M_0 y(t) + \alpha_1 M_1 y(t) = s_1(-\lambda_1) y(t) + 1\cdot \varphi(t)y^\Delta(t) = \left(\varphi D-\lambda_1I\right)y(t) $$
and the result holds. Assume \eqref{factorize} holds for $n\ge 1$. Then we have $\alpha_{n+1}\equiv 1$ and
\begin{eqnarray*}
 \sum_{k=0}^{n+1} \alpha_k M_k y(t) &=& \alpha_0 y(t) + \sum_{k=1}^n \alpha_k M_k y(t) + M_{n+1}y(t) \\
 &=& s_{n+1}^{n+1} y(t) + \sum_{k=1}^n s_{n+1-k}^{n+1} M_k y(t) + \varphi(t)\left(M_{n}y\right)^\Delta(t) \\
 &=& -\lambda_{n+1}s_n^n y(t) + \sum_{k=1}^n \left(s_{n+1-k}^{n}-\lambda_{n+1}s_{n-k}^n\right)M_k y(t) + \varphi(t)D\left(M_{n}y\right)(t) \\
 &=& -\lambda_{n+1}\left[s_n^n y(t) + \sum_{k=1}^n s_{n-k}^n M_k y(t)\right] + \sum_{k=1}^n s_{n+1-k}^{n}M_k y(t) + \varphi(t)D\left(M_{n}y\right)(t) \\
 &=& -\lambda_{n+1}\sum_{k=0}^n s_{n-k}^n M_k y(t) + \varphi(t)D\left(\sum_{k=1}^n s_{n+1-k}^{n}M_{k-1} y(t) + M_{n}y\right)(t) \\
 &=& -\lambda_{n+1}\sum_{k=0}^n s_{n-k}^n M_k y(t) + \varphi(t)D\left(\sum_{k=0}^{n-1} s_{n-k}^{n}M_{k} y(t) + M_{n}y\right)(t) \\
 &=& -\lambda_{n+1}\sum_{k=0}^n s_{n-k}^n M_k y(t) + \varphi(t)D\sum_{k=0}^{n} s_{n-k}^{n}M_{k} y(t) \\
 &=& \left(\varphi(t)D-\lambda_{n+1}I\right)\sum_{k=0}^n s_{n-k}^n M_k y(t) \\
 &=& \left(\varphi(t)D-\lambda_{n+1}I\right)\sum_{k=0}^n \alpha_{k} M_k y(t) \\
 &=& \left(\varphi(t)D-\lambda_{n+1}I\right)\prod_{k=1}^{n} \left(\varphi D-\lambda_kI\right)y(t)
\end{eqnarray*}
and the proof is complete.
\end{proof}


\begin{theorem}[Hyers-Ulam Stability]\label{nonhomodeltadelta}
Given $y,\varphi,f\in\czero[a,b]_\T$ with $|\varphi|\ge A>0$ for some constant $A$, and $\alpha_k\in\R$ with $\alpha_n\equiv 1$, consider \eqref{2ndivc} with $M_ky\in\crdone[a,b]{_\T}$ for $k=0,\cdots,n-1$. 
Using the $\lambda_k$ from the factorization in Lemma $\ref{lemma2.3}$, assume 
\begin{equation}\label{CEregress} 
 \varphi(t)+\lambda_k \mu(t) \ne 0, \quad k=1,2,\cdots,n
\end{equation}
for all $t\in[a,\sigma^{n-1}(b)]_\T$. Then \eqref{2ndivc} has Hyers-Ulam stability on $[a,b]_\T$.
\end{theorem}

\begin{proof}
Let $\varepsilon>0$ be given, and suppose there is a function $x$, with $M_kx\in\crdone[a,b]{_\T}$, that satisfies
$$ \left|  \sum_{k=0}^n \alpha_k M_k x(t) - f(t) \right|\le \varepsilon, \quad t\in[a,b]_{\T}. $$
We will show there exists a solution $u$ of \eqref{2ndivc} with $M_ku\in\crdone[a,b]{_\T}$ for $k=0,1,\cdots,n-1$ such that $|x-u|\le K\varepsilon$ on $[a,\sigma^n(b)]_\T$ for some constant $K>0$.

To this end, set
\begin{eqnarray*}
 g_1 &=& \varphi x^\Delta-\lambda_1 x = \left(\varphi D-\lambda_1 I \right)x \\
 g_2 &=& \varphi g_1^\Delta-\lambda_2 g_1 = \left(\varphi D-\lambda_2 I \right)g_1 \\
 \vdots & & \\
 g_k &=& \varphi g_{k-1}^\Delta-\lambda_k g_{k-1} = \left(\varphi D-\lambda_k I \right)g_{k-1} \\
 \vdots & & \\
 g_n &=& \varphi g_{n-1}^\Delta-\lambda_n g_{n-1} = \left(\varphi D-\lambda_n I \right)g_{n-1}.
\end{eqnarray*}
This implies by Lemma \ref{lemma2.3} that
$$ g_n(t)-f(t) = \sum_{k=0}^n \alpha_k M_k x(t) - f(t), $$
so that
$$ \left| g_n(t) - f(t) \right|\le \varepsilon, \quad t\in[a,b]_{\T}. $$
By the construction of $g_n$ we have $\left|\varphi g_{n-1}^\Delta-\lambda_n g_{n-1}-f \right| \le \varepsilon$, that is
$$ \left|g_{n-1}^\Delta-\frac{\lambda_n}{\varphi} g_{n-1}-\frac{f}{\varphi} \right| \le \frac{\varepsilon}{|\varphi|} \le \frac{\varepsilon}{A}. $$
By \cite[Lemma 2.3]{and} and \eqref{CEregress} there exists a solution $w_1\in\crdone[a,b]{_\T}$ of
\begin{equation}\label{w1sol}
 w^\Delta(t) - \frac{\lambda_n}{\varphi(t)}w(t)-\frac{f(t)}{\varphi(t)}=0, \quad\text{or}\quad \varphi(t) w^\Delta(t) - \lambda_n w(t)-f(t)=0,
\end{equation}
$t\in[a,b]_\T$, where $w_1$ is given by
$$ w_1(t) = e_\frac{\lambda_n}{\varphi}(t,\tau_1)g_{n-1}(\tau_1)+\int_{\tau_1}^t e_\frac{\lambda_n}{\varphi}(t,\sigma(s))\frac{f(s)}{\varphi(s)}\Delta s, \quad\text{any}\quad \tau_1\in[a,\sigma(b)]_{\T}, $$
and there exists an $L_1>0$ such that
$$ |g_{n-1}(t)-w_1(t)| \le L_1 \varepsilon/A, \quad t\in[a,\sigma(b)]_\T. $$
Since $g_{n-1} = \varphi g_{n-2}^\Delta-\lambda_{n-1} g_{n-2}$, we have that
$$ |\varphi g_{n-2}^\Delta-\lambda_{n-1} g_{n-2}(t)-w_1(t)| \le L_1 \varepsilon/A, \quad t\in[a,\sigma(b)]_\T. $$
Again we apply \cite[Lemma 2.3]{and} to see that there exists a solution $w_2\in\crdone[a,\sigma(b)]{_\T}$ of
$$ w^\Delta(t) - \frac{\lambda_{n-1}}{\varphi(t)}w(t)-\frac{w_1(t)}{\varphi(t)}=0, \quad\text{or}\quad \varphi(t) w^\Delta(t) - \lambda_{n-1} w(t)-w_1(t)=0, $$
$t\in[a,\sigma(b)]_\T$, where $w_2$ is given by
$$ w_2(t) = e_\frac{\lambda_{n-1}}{\varphi}(t,\tau_2)g_{n-2}(\tau_2)+\int_{\tau_2}^t e_\frac{\lambda_{n-1}}{\varphi}(t,\sigma(s))\frac{w_1(s)}{\varphi(s)}\Delta s, \quad\text{any}\quad \tau_2\in[a,\sigma^2(b)]_{\T}, $$
and there exists an $L_2>0$ such that
$$ |g_{n-2}(t)-w_2(t)| \le L_2L_1 \varepsilon/A^2, \quad t\in[a,\sigma^2(b)]_\T. $$

Continuing in this manner, we see that for $k=1,2,\cdots,n-1$ there exists a solution $w_k\in\crdone[a,\sigma^{k-1}(b)]{_\T}$ of
$$ w^\Delta(t) - \frac{\lambda_{n-k+1}}{\varphi(t)}w(t)-\frac{w_{k-1}(t)}{\varphi(t)}=0, \quad\text{or}\quad \varphi(t) w^\Delta(t) - \lambda_{n-k+1} w(t)-w_{k-1}(t)=0, $$
$t\in[a,\sigma^{k-1}(b)]_\T$, where $w_k$ is given by
\begin{equation}\label{wkform}
 w_k(t) = e_\frac{\lambda_{n-k+1}}{\varphi}(t,\tau_k)g_{n-k}(\tau_k)+\int_{\tau_k}^t e_\frac{\lambda_{n-k+1}}{\varphi}(t,\sigma(s))\frac{w_{k-1}(s)}{\varphi(s)}\Delta s, \quad\text{any}\quad \tau_k\in[a,\sigma^k(b)]_{\T}, 
\end{equation}
and there exists an $L_k>0$ such that
$$ |g_{n-k}(t)-w_k(t)| \le \prod_{j=1}^{k}L_j \varepsilon/A^k, \quad t\in[a,\sigma^k(b)]_\T. $$
In particular, for $k=n-1$, 
$$ |g_{1}(t)-w_{n-1}(t)| \le \prod_{j=1}^{n-1}L_j \varepsilon/A^{n-1}, \quad t\in[a,\sigma^{n-1}(b)]_\T $$
implies by the definition of $g_1$ that 
$$ \left| x^\Delta(t)-\frac{\lambda_1}{\varphi(t)}x(t)-\frac{w_{n-1}(t)}{\varphi(t)} \right| \le \prod_{j=1}^{n-1}L_j \varepsilon/A^{n}, \quad t\in[a,\sigma^{n-1}(b)]_\T. $$
Thus there exists a solution $w_n\in\crdone[a,\sigma^{n-1}(b)]{_\T}$ of
$$ w^\Delta(t) - \frac{\lambda_{1}}{\varphi(t)}w(t)-\frac{w_{n-1}(t)}{\varphi(t)}=0, \quad\text{or}\quad \varphi(t) w^\Delta(t) - \lambda_{1} w(t)-w_{n-1}(t)=0, $$
$t\in[a,\sigma^{n-1}(b)]_\T$, where $w_n$ is given by
\begin{equation}\label{wnform}
 w_n(t) = e_\frac{\lambda_{1}}{\varphi}(t,\tau_n)x(\tau_n)+\int_{\tau_n}^t e_\frac{\lambda_{1}}{\varphi}(t,\sigma(s))\frac{w_{n-1}(s)}{\varphi(s)}\Delta s, \quad\text{any}\quad \tau_n\in[a,\sigma^n(b)]_{\T}, 
\end{equation}
and there exists an $L_n>0$ such that
\begin{equation}\label{xminusw}
 |x(t)-w_n(t)| \le K\varepsilon:=\prod_{j=1}^{n}L_j \varepsilon/A^n, \quad t\in[a,\sigma^n(b)]_\T. 
\end{equation}
By construction, 
\begin{eqnarray*}
 \left(\varphi D-\lambda_1 I\right)w_n(t) &=& w_{n-1}(t) \\
 \prod_{k=1}^{2}\left(\varphi D-\lambda_k I\right)w_n(t) &=& \left(\varphi D-\lambda_2 I\right)w_{n-1}(t) = w_{n-2}(t) \\
 \vdots & & \\
 \prod_{k=1}^{n}\left(\varphi D-\lambda_k I\right)w_n(t) &=& \left(\varphi D-\lambda_n I\right)w_{1}(t) \stackrel{\eqref{w1sol}}{=} f(t)
\end{eqnarray*}
on $[a,\sigma^{n-1}(b)]_\T$, so that $u=w_n$ is a solution of \eqref{2ndivc}, with $u\in\crdone[a,\sigma^{n-1}(b)]{_\T}$ and $|x(t)-w_n(t)| \le K\varepsilon$ for $t\in[a,\sigma^n(b)]_\T$ by \eqref{xminusw}. Moreover, using \eqref{wnform} and \eqref{wkform}, we have an iterative formula for this solution $u=w_n$ in terms of the function $x$ given at the beginning of the proof.
\end{proof}


\end{document}